\documentclass[12pt]{amsart}

\usepackage[utf8]{inputenc}
\usepackage[T1]{fontenc}      
\usepackage{amsfonts}     
\usepackage{amsthm}   
\usepackage{ amssymb }      
\usepackage{amsmath}
\usepackage{bm}
\usepackage{mathrsfs}		

\def\Z{\mathbb Z}
\def\G{\mathbb G}
\def\P{\mathbb P}
\def\Q{\mathbb Q}
\def\R{\mathbb R}
\def\HS{\mathscr V}
\def\L{\Lambda}
\def\x{\bm{x}}
\def\e{\bm{e}}
\def\m{\bm{m}}

\def\b{\bm{b}}
\def\y{\bm{y}}
\def\t{\bm{t}}
\def\B{\bm{B}}

\DeclareMathOperator{\disc}{disc}
\DeclareMathOperator{\diag}{diag}
\DeclareMathOperator{\lcm}{lcm}

\newtheorem{lemma}{Lemma}[section]

\newtheorem{theorem}{Theorem}[section]

\theoremstyle{definition}

\newtheorem{remark}{Remark}[section]

\numberwithin{equation}{section}

\title[Rational points on quadrics]{A uniform estimate for the density of rational points on quadrics}      
\author{F\'elicien Comtat} 
        
\begin{document} 

\begin{abstract}
This paper is concerned with the density of rational points of bounded height lying on a variety defined by an integral quadratic form $Q$.
In the case of four variables, we give an estimate that does not depend on the coefficients of $Q$.
For more variables, a similar estimate still holds with the restriction that we only count points which do not lie on $\Q$-lines.
\end{abstract}

\maketitle

\section{Introduction}

Given a non-singular quadratic form $Q \in \Z[x_1, \dots ,x_n]$, we are interested in the asymptotic distribution of rational points lying on the projective hypersurface $\HS$ defined by $Q=0$.
More precisely,
define $N(Q, B)$ be the number of primitive points $\x \in \Z^n$ with $|x_i| \le B$ for each $i$ and such that $Q(\x)=0$.

In the case $n=4$, it can be deduced from the work of Heath-Brown \cite[Theorems 6,7]{CircleMethod} that
$$
N(Q,B) \sim \left\{
    \begin{array}{ll}
       c_Q B^2& \mbox{if } \disc(Q) \text{ is not a square,}\\
       c_Q B^2 \log(B) & \mbox{otherwise,}
    \end{array}
\right.
$$
where $c_Q$ is a product of local densities that depends on $Q$. However, in some applications one may be interested in estimates which are uniform with respect to the coefficients of $Q$. 
In this spirit, if $Q$ is a quadratic form of rank at least $3$ in $n$ variables, then for any $\epsilon > 0$ we have  \cite[Theorem 2]{density}
$$N(Q,B) \ll_\epsilon B^{2+\epsilon},$$
where the implied constant only depends on $\epsilon$. In the case of four variables, we show that this $\epsilon$ can be removed, in the following sense.

\begin{theorem}\label{bhb}
Let $Q$ be a non-singular integral quadratic form in four variables. Then
$$
N(Q,B) \ll \left\{
    \begin{array}{ll}
       B^2& \mbox{if } \disc(Q) \text{ is not a square,}\\
       B^2 \log(B) & \mbox{otherwise,}
    \end{array}
\right.
$$
the implied constant being absolute. 
\end{theorem}

Theorem \ref{bhb} confirms a conjecture made by Browning and Heath-Brown in \cite{2018}. In fact, the main result in \cite{2018} gives an estimate for 
$N(Q,B)$ which actually gets sharper than Theorem \ref{bhb} for suitably generic quaternary quadratic forms (namely those for which the discriminant is close to being square-free
and of order $\| Q \|$).
 
When $n=4$ and $\HS(\Q) \neq \varnothing$, the assumption that the surface $\HS$ contains $\Q$-lines is equivalent to $\disc(Q)$ being a square. 
Thus one may estimate separately the contribution of points which do not lie on such lines and the contribution from the $\Q$-lines. 

To this end, it is also useful to have precise estimates for the distribution of rational points in arbitrary boxes instead of hypercubes only.
Namely, for $B_1,\dots , B_n \ge 1$, let $N(Q,\B)$ be the number of primitive points on the hypersurface $\HS$ with $|x_i| \le B_i$ for each $i$,
and let $N_1(Q,\B)$ be the number of points $\x$ counted by $N(Q, \B)$ with the additional condition that $\x$ does not lie on a $\Q$-line included in $\HS$.
Work of Browning and Heath-Brown \cite{BigArticle} gives precise estimates for these counting functions. 
By revisiting their argument we shall prove the following result.

\begin{theorem}\label{1.2}
Let $n \ge 3$, and $B_1, \dots , B_n \ge 1$ and define $V= \prod_{i=1}^n B_i$. 
Let $Q \in \Z[x_1, \dots, x_n]$ be any quadratic form.
Then $N_1(Q,\B) \ll V^{\frac{n-2}{n}}$, where the implied constant depends only on $n$.
\end{theorem}

Theorem \ref{1.2} is a refinement of \cite[Theorem 2]{BigArticle} in the particular case of quadratic forms, as we get rid of all the logarithmic factors. 
Combined with an estimation of the contribution from the lines separately, it will prove a vital tool in the proof of Theorem \ref{bhb}.

\section{Preliminaries}

In order to estimate the counting functions we are interested in, we shall need to deal with various sublattices of $\Z^n$. 
We therefore introduce for any such lattice $\L$, and $B_1,\dots , B_n \ge 1$
$$P(Q,\B,\L)=\{[\x] : \x \in \L, Q(\x)=0, \forall i |x_i| \le B_i \},$$
where $[\x]$ is the image in $\P^{n-1}(\R)$ of $\x$. We also define
$$S(Q,\B,\L)=\left\{\x \in \L :  c_{\L}(\x)=1, Q(\x)=0, \forall i |x_i| \le B_i \right\},$$
where $c_{\L}(\x)$ denotes the gcd of the coefficients of $\x$ in any base of $\L$, and
$$Z(Q,\B,\L)=\{\x \in \L :  \gcd(x_1, \dots, x_n)=1, \forall i |x_i| \le B_i \}.$$
There is an obvious 2-to-1 correspondence between $S(Q,\B,\L)$ and $P(Q,\B,\L)$, and when $\L=\Z^n$ we have $Z(Q,\B,\L)=S(Q,\B,\L)$.
However in general there is only an inclusion of $Z(Q,\B,\L)$ into $S(Q,\B,\L)$.
When $B_1= \dots =B_n=B$, we shall denote these sets by $P(Q,B,\L)$, $S(Q,B,\L)$ and $Z(Q,B,\L)$, respectively. 

We begin with recalling some useful facts about lattices. Given any lattice $\L \subset \Z^n$ of dimension $r$, its \textit{determinant} is the $r$-dimensional volume of any fundamental parallelepiped of
 $\L$. If $M$ is the $n \times r$ matrix whose columns are the vectors of any basis of $\L$, then we have 
 $$(\det \L)^2=\det(M{}^TM).$$
 It follows that if $\L_1 \subset \L_2$ have the same dimension then $\det \L_2 \mid \det \L_1$.
 
 Let $\m_1$ be any shortest non-zero vector in $\L$, and for $i < r$, define $\m_{i+1}$ as any shortest vector in $\L$ not contained in the span of $\m_1, \dots, \m_i$. We obtain a
  so-called \textit{minimal basis} $\m_1, \dots, \m_r$ of $\L$, and the corresponding $s_i=\|\m_i\|$ are the \textit{successive minima} of $\L$ with respect to the Euclidean length. Then we have 
 \begin{equation}
 \prod_{i=1}^r s_i \ll \det \L \le \prod_{i=1}^r s_i.
 \end{equation}
 In addition, we have the following useful lemma (see \cite[Lemma 5]{Davenport}).
 \begin{lemma}\label{basis}
 Let $\L \subset \Z^n$ be a lattice of dimension $r$, with successive minima $s_1 \le \dots \le s_r$ and let $\m_1, \dots ,\m_r$ be a minimal basis of $\L$. Then for any 
 $$\x = \sum_{i=1}^r \lambda_i \m_i \in \L,$$
 we have $\lambda_i \ll \frac{\|\x\|}{s_i}$ for all $i \le r$.
 \end{lemma}
 
 The next result, which appears in the proof of Theorem 4 in \cite{density}, shows that we can restrict attention to quadratic forms $Q$ of height bounded in terms of $B$.
 \begin{lemma}\label{trick}
 Let $Q \in \Z[x_1, \dots, x_n]$ be a quadratic form and denote by $\|Q\|$ the maximum of the absolute values of the coefficients of $Q$. 
 Then for any $B_1 \le \dots \le B_n$ there exists a quadratic form $G \in \Z[x_1, \dots, x_n]$ with 
 $\|G\| \ll B_n^{n(n+1)}$ such that $P(Q,\B,\Z^n) \subset P(G,\B,\Z^n)$.
 \end{lemma}
 
 \section{Proof of Theorem \ref{1.2}}\label{pt1}
 
 We proceed by induction on $n$.
 The base case $n=3$ is Theorem 6 of \cite{BigArticle}. Note that if $Q$ is singular, then all points lie on a line, therefore in this case the claim of Theorem \ref{1.2} is automatically verified. 

 Consider now $n \ge 4$ and assume $Q$ is a fixed non-singular integral quadratic form in $n$ variables.
 Let $B_1 \le B_2 \le \dots \le B_n$ (this condition is not restrictive since we can permute variables) and $V=\prod_{i=1}^n B_i$.
 We may assume without loss of generality that the $B_i$'s are powers of two.
By Lemma \ref{trick} we may furthermore assume without loss of generality that $Q$ is primitive and $\|Q\| \le B_n^{n(n+1)}$. 
 
 For any prime $p$ we introduce 
 $$S(Q,\B,\L,p) = \{\x \in S(Q,\B,\L) : p \nmid \nabla Q(\x) \}$$
 and
 $$Z(Q,\B,\L,p) = \{\x \in Z(Q,\B,\L) : p \nmid \nabla Q(\x) \}.$$
 We will strongly rely on the following result from \cite[Lemma 8]{BigArticle}.
 
 \begin{lemma} \label{key}
 Let $\mathcal B \ge 1$ and $\L \subset \Z^n$ be a lattice of dimension $r \ge 2$, with largest successive minimum $\ll \mathcal B$. 
 Let $p$ be a prime not dividing $\det \L$. 
 Then there is an integer $I \ll p^{r(r-2)}\left(\frac{ \mathcal B^r}{\det \L}\right)^{\frac{r-2}r}$ and lattices $\L_1, \dots, \L_I \subset \L$ of dimension $r-1$ such that
 $$S(Q, \mathcal B,\L,p) \subset \bigcup_{i=1}^I \L_i.$$
 For any $i$, the successive minima of $\L_i$ are all $O(p^r \mathcal B)$. Moreover for any $\alpha > 0$ we have 
 \begin{equation}
 \label{c1}
 \#\{ i : \det \L_i \le \alpha \det \L\} \ll p^{r-2}(\alpha \mathcal B)^{\frac{r-2}r}.
 \end{equation}
 \end{lemma}
 
 \begin{remark}
 What is denoted by $S(Q, \mathcal B,\L)$ in \cite{BigArticle} is what we denote here by $Z(Q, \mathcal B,\L)$. 
 Thus it might seem that we should have replaced $S(Q, \mathcal B,\L,p)$ with $Z(Q, \mathcal B,\L,p)$ in the above statement.
 However, going through the proof, one can ensure this is what is proven in \cite{BigArticle}. 
 Indeed, the authors proceed by sending the lattice $\L$ to $\Z^n$, then to embed the primitive points on the accordingly transformed quadric into a collection
 of sublattices, and eventually take back the image of these sublattices in $\L$. The point is that not only vectors from $Z(Q, \mathcal B,\L)$, but actually from 
 $S(Q, \mathcal B,\L)$ (with our notation), have an image in $\Z^n$ that is primitive.
 Another slight difference is that in the original formulation of  \cite{BigArticle}, it is assumed that $\alpha \ge 1$. However, here again this condition can be relaxed.
 The construction is such that to each lattice $\L_i$ is assigned a "depth" $k$. Moreover, there are at most $p^{k(r-2)}$ lattices of depth $k$, and 
 $\L_i$ has determinant at least $p^{(k-1)r} \mathcal B^{-1}\det \L$. Thus, for any $\alpha > 0$, $\det \L_i \le \alpha \det \L$ implies $p^k \ll p (\alpha \mathcal B)^{1/r}$. 
 Therefore there are at most $p^{r-2} (\alpha \mathcal B)^{\frac{r-2}r}$ such lattices $\L_i$.

 \end{remark}
 
 The general idea of the proof of this lemma is to consider the reduction of $Q=0$ modulo $p^k$ for some appropriate exponents $k$, 
 and to perform a Taylor expansion of order $2$ in order to embed the lifts in $\L$ of every projective point $[\x]$ modulo $p^k$ in a corresponding lattice $M_{[\x]} \subset \L$ of dimension $r$. 
 Moreover, for $k$ big enough these lattices $M_{[\x]}$ have largest successive minimum $\gg \mathcal B$.
 Hence, by Lemma \ref{basis}, under the assumption that $|\x| \le \mathcal B$, one can restrict to the sublattice of $M_{[\x]}$ generated by the first $r-1$ vectors of a minimal basis. 
 On the other hand, the larger $k$ is, the more points modulo $p^k$ have to be considered. Thus one has to choose $k$ carefully.
   
  In order to come down from $S(Q, \mathcal B,\L)$ to $S(Q,\mathcal B,\L,p)$ one has to find a set of primes $p$ in a way that ensures 
  that for each $\x \in S(Q, \mathcal B,\L)$, at least one of these $p$ does not divide $\nabla Q(\x)$.
  To this end, $p$ has to range in some interval depending on $\|Q\| \mathcal B$, and that is where the logarithmic factors enter into the picture in \cite{BigArticle}. 
  In our work, the main idea is to exploit the situation $p \mid \nabla Q(\x)$ as an additional constraint on $\x$.

 For any positive integer $m$, let $p_m$ denote the $m$-th prime number. 
 For simplicity, in the sequel we shall use the notation $q_m = p_2 \dots p_{m-1}$ for the product of the first $m-2$ odd primes. 
 For any lattice $\L \subset \Z^n$, denote by 
 $$\L(m) = \{ \x \in \L : q_m \mid \nabla Q(\x) \}.$$
 Since $Q$ is quadratic, $\nabla Q$ is linear and $\L(m)$ is a lattice. 
 Let $M$ be the half-integer symmetric matrix corresponding to $Q$, so that $Q(\x)={}^T\x M \x$. 
 Write $2M$ under Smith normal form: $2M=SDT$ with $S$ and $T$ unimodular,
 and $D=\diag(d_i)$,with  $d_i \mid d_{i+1}$ for $1 \le i \le n$. Hence $d_1 \in \{1,2\}$ since $Q$ is primitive.
 
 \begin{lemma}\label{deter}
 Let $b_1, \dots ,b_n$ be powers of two, $\L_{\b}= \prod_{i=1}^n (b_i \Z)$, and $v=\prod_{i=1}^n b_i$. 
 Then $$\det \L_{\b}(m) = v \prod_{i=1}^n a_i$$
 with $a_i = \frac{q_m}{\gcd(q_m,d_i)}$.
 In particular, it is divisible by $q_m$ since $a_1=q_m$.
 \end{lemma}
 
 \begin{proof}
 Write $\L_{\b}(m)=\Z^n(m) \cap \L_{\b}$. Obviously, $\L_{\b}$  has determinant $v$. 
 It follows that 
 $\lcm(\det \Z^n(m), v) \mid \det \L_{\b}(m).$ We have
 \begin{align*}
 \Z^n(m) &=  \{ \x \in \Z^n : q_m \mid \nabla Q(\x) \} \\
 				 &=  \{ \x \in \Z^n : q_m \mid SDT\x \} \\
 				 &=  \{ T^{-1}\y : \y \in T\Z^n : q_m \mid SD \y \} \\
 				 &= \{ T^{-1}\y : \y \in \Z^n : q_m \mid D \y \}
 \end{align*}
 since $S,T$ are unimodular. Therefore a basis for $\Z^n(m)$ is given by 
 $$T^{-1}\left(a_1 \e_1, \dots, a_n \e_n\right),$$ 
 where $(\e_1,  \dots , \e_n)$ is the canonical basis of $\Z^n$. 
 In particular 
 $$\det \Z^n(m) = \prod_{i=1}^n a_i$$
  and the above least common multiple is the product of both arguments since $v$ is a power of two and $\det \Z^n(m)$ is odd. 
 On the other hand, $\L_b(m)$ contains the image of $\Z^n(m)$ by \begin{align*}
 \phi: \Z^n & \to \Z^n \\
 \x & \mapsto \left(b_1x_1, \dots ,b_nx_n\right),
 \end{align*}
  so $\det \L_b(m) \mid v \prod_{i=1}^n a_i$ and we are done.
 \end{proof}
 
 For any non-singular point $\x$ on the hypersurface $\HS$, let $m \ge 2$ such that $p_m$ is the smallest odd prime not dividing $\nabla Q(\x)$. 
 Then $\x \in \Z^n(m)$ and we can write 
 $$Z(Q,\B,\Z^n)=\bigcup_m Z(Q,\B,\Z^n(m),p_m).$$
 Note that if $Z(Q,\B,\Z^n(m),p_m)$ is non empty then there exists a primitive vector $\x$ with $|\x| \le B_n$ and $q_m \le |\nabla Q(\x)| \ll \| Q \| |\x| \ll B_n^{n^2+n+1}$ by assumption.
 Hence we must have $m \ll \log(B_n)$.
 
 Next consider
 \begin{align*}
 \phi: \Z^n & \to \Z^n \\
 \x & \mapsto \left(\frac{V}{B_1}x_1, \dots ,\frac{V}{B_n}x_n\right),
 \end{align*}
 which sends the box $\B$ to the centered hypercube of side length $V$.
 Then for any $\x \in \Z^n$ we have $Q(\x)=0$ if and only if $Q( \phi^{-1}(\y))=0$, where $\y=\phi(\x)$.
 Therefore, the zeros of $Q$ correspond under this linear transformation to those of the form $Q_{\phi}=V^2 Q \circ \phi^{-1}$, 
 where the scaling factor $V^2$ is only there to ensure that the coefficients are integers. 
 Of course, the image by $\phi$ of a primitive point need not be primitive any more.
 However if $\x$ and $\x'$ are two distinct primitive points in the box $\B$ then $\phi(\x)$ and $\phi(\x')$ define two distinct projective points of height at most $V$. 
 Clearly, $\nabla Q_{\phi}(\y)=V^2 \phi^{-1}(\nabla Q(\x))$. Thus for any lattice $\L$ and any odd prime $p$, we have 
 $$\#S(Q,\B,\L,p) \le \#S(Q_{\phi},V,\phi(\L),p).$$
  
 In view of what precedes, in the notation of Lemma \ref{deter} we have 
 $$\# Z(Q,\B,\Z^n) \le \sum_m \#S(Q_{\phi},V,\L_{\b}(m),p_m),$$
 with $b_i=\frac{V}{B_i}$.
 For future use we note that for this choice of $\b$, Lemma \ref{deter} implies 
 \begin{equation} \label{minor}
 \det \L_{\b}(m) \gg V^{n-1} q_m.
 \end{equation}
 For each $m$, two cases can occur: if $\L_{\b}(m)$ has largest successive minimum $\gg V$, 
 then by Lemma \ref{basis} any point of height at most $V$ contained in $\L_{\b}(m)$ is actually contained in 
 the $(n-1)$-dimensional sublattice of $\L_{\b}(m)$ generated by the $n-1$ first vectors of a minimal basis.
 Otherwise Lemma \ref{key} applies with $\L=\L_{\b}(m)$ and $\mathcal B=V$. 
 
In the first case, taking the preimage by $\phi$, each point in the box $\B$ contained in $\Z^n(m)$ is actually contained in some $(n-1)$-dimensional sublattice $\L$ of $\Z^n(m)$. 
It means that one of the coordinates can be expressed as a linear function of the other ones, say $x_1 = \ell(x_2, \dots ,x_n)$. 
Furthermore, $\ell$ is defined over $\Q$. Therefore for some integer $d$, the form $Q_\ell=dQ(\ell(x_2, \dots ,x_n),x_2, \dots ,x_n)$ has integral coefficients. 
Then 
we have
\begin{align*}
\#S(Q,\B,\L, p) &\le \#S(Q,\B,\Z^n)\\
								&\le \# S(Q_\ell,\B',\Z^{n-1}),
\end{align*}
where $\B'$ is the box  with sides of length $B_2, \dots, B_n$.
By the induction hypothesis, the number of points not contained in a line coming from $\Z^n(m)$ is in this case $ \ll (B_2 \dots B_n)^{\frac{n-3}{n-1}}$.
Since $m \ll \log(B_n)$, the total contribution of this case is $$\ll(B_2 \dots B_n)^{\frac{n-3}{n-1}} \log(B_n) \ll (B_1 \dots B_n)^{\frac{n-2}{n}}$$
 since $B_1 \gg 1$ and $\log(B_n) \ll B_n^{\frac{2}{n(n-1)}}$.
 
 Consider now the second case. 
 Let $\L_i \subset \L_{\b}(m) \subset \Z^n(m)$ be any sublattice of dimension $n-1$ arising from Lemma \ref{key},
 with minimal basis $(\t_1,  \dots ,\t_{n-1})$ and successive minima $s_1, \dots , s_{n-1}$ say. 
 Set 
\begin{align*}
\psi : \Z^{n-1} & \to \L_i \\
\x & \mapsto \sum_{j=1}^{n-1} x_j \t_j.
\end{align*}
By Lemma \ref{basis}, we are left with counting projective points $\x$ of $\Z^{n-1}$ with $x_j \le \frac{V}{s_j}$, such that $Q(\psi(\x))=0$, and not lying on any line.
However, to apply the induction hypothesis, we must ensure that the sides of the box we consider are larger than $1$. 
We know that the successive minima of $\L_i$ are all $ \le  c p_m^nV$ for some constant $c$ depending only on $n$.
We shall rather count those points $\x$ in the larger box $x_j \le \frac{V_0}{s_j}$, where $V_0= c p_m^nV$.
By the induction hypothesis, we find 
\begin{align*}
& \ll  \left(\frac{V_0}{s_1} \dots \frac{V_0}{s_{n-1}} \right)^{(n-3)/(n-1)}\\
& \ll \left(\frac{V_0^{n-1}}{\det \L_i}\right)^{(n-3)/(n-1)}
\end{align*}
such points for each sublattice $\L_i$.
By (\ref{c1}), for any $\alpha > 0$, the lattices $\L_i$ such that
\begin{equation}
\label{dyadic}
\frac{\alpha}{2} \det \L_{\b}(m)< \det \L_i \le \alpha \det \L_{\b}(m)
\end{equation}
yield a total contribution of at most
$$p_m^{n-2}(\alpha V)^{\frac{n-2}n}\left(\frac{V_0^{n-1}}{\alpha \det \L_{\b}(m)}\right)^{(n-3)/(n-1)}.$$
Since 
the successive minima of $\L_i$ are all $\le V_0$, we have 
$$
\det \L_i \ll V_0^{n-1}.$$
Hence for each possible $i$, $\det \L_i$ has to lay in one of the dyadic intervals (\ref{dyadic}) for some $\alpha$ in the 
 range
$$\alpha \ll \frac{V_0^{n-1}}{\det \L_{\b}(m)} \doteq X,$$
say. Summing for each $m$ the contributions of these dyadic intervals, we get at most
$$p_m^{n-2}V^{\frac{n-2}n} X^{\frac{n-3}{n-1}} \sum_{2^k \ll X} 2^{k \left( \frac{n-2}n - \frac{n-3}{n-1}\right) }
\ll p_m^{n-2}V^{\frac{n-2}n}  X^{\frac{n-2}n}.$$
By (\ref{minor}), we have 
$$X \ll \frac{p_m^{n(n-1)}}{q_m},$$
 thus the total contribution from each $m$ is
$$\ll p_m^{n(n-2)}\left(\frac{V}{q_m}\right)^{\frac{n-2}{n}}.$$
But $p_m^{a}q_m^{-b} \ll (m \log m)^a 2^{-b m}$ for any $a, b > 0$ by the prime number  theorem. Hence the series 
\begin{equation} \label{series}
\sum_{m}p_m^{n(n-2)}q_m^{-\frac{n-2}{n}}
\end{equation} converges.
Adding up all these contributions, we therefore get $\ll V^{\frac{n-2}{n}}$ points, which finally establishes Theorem \ref{1.2}. 

\section{Proof of Theorem \ref{bhb}}

We have to add to the estimate given by Theorem \ref{1.2} the contribution of points lying on the various lines included in the surface $Q=0$. 

The argument involved for the case of four variables is quite similar to the last section of \cite{2018},
with a slight modification due to the fact that we do not consider the same collection of planes, so we do not get the same number of lines.
Consider the planes generated by the various sublattices of dimension $3$ that we constructed above.
Recalling that we applied Lemma \ref{key} for each $m$ with $\mathcal B = V$ and $\L=\L_{\b}(m)$,
we have at most 
$$I \ll p_m^{8}\left(\frac{V^4}{ \det \L_{\b(m)}}\right)^{\frac12} \ll p_m^{8}\left(\frac{V}{ q_m}\right)^{\frac12}$$
such hyperplanes. 
Hence using again the fact that the series (\ref{series}) converges, the total number of planes is $O\left(V^{\frac12}\right) =O \left(B^{2} \right)$.
But the intersection of each plane with the surface $\HS$ has at most two irreducible components, which may be a line or a conic.
Hence we have the contribution of at most $N \ll B^{2}$ lines to take into account. 
The integral points on each line $L$ form a certain sublattice $\Lambda_L$ of dimension 2. 
By Lemma \ref{basis}, this lattice contributes $O\left(1+ \frac{B^2}{\det \Lambda_L} \right)$.
But $L$ also corresponds to a point of height $H(L)=\det \Lambda_L$ in the variety $F_1(Q) \subset \G(1,n-1)$ of lines included in $Q=0$.
Order the lines $L_1, \dots, L_N$ by non-decreasing height. Since, by the work of Walsh \cite{Walsh}, $F_1(Q)$ has $O(H)$ points of height at most $H$ with an absolute implied constant, we get that 
$$n \le \# \{L_i : H(L_i) \le H(L_n)\} \ll H(L_n).$$
Therefore the total contribution from the $\mathbb Q$-lines is 
$$\ll \sum_{n \le N} 1 + \frac{B^2}{H(L_n)} \ll N+B^2 \log(N) \ll B^2 \log(B),$$
which establishes Theorem \ref{bhb}. 

 \bibliographystyle{plain}
\bibliography{biblio}

\end{document}